\documentclass[oneside,12pt]{amsart} \topmargin      -15mm \textwidth      160 true mm
\usepackage{amsmath, amssymb, amsfonts, amstext, amsthm}
\usepackage[mathscr]{euscript}
\usepackage{enumerate, mathtools}
\theoremstyle{plain}
\newtheorem{theorem}{Theorem}[section]
\newtheorem{proposition}[theorem]{Proposition}
\newtheorem{lemma}[theorem]{Lemma}
\newtheorem{corollary}[theorem]{Corollary}
\theoremstyle{remark}
\newtheorem{definition}[theorem]{Definition}
\newtheorem{eg}[theorem]{Example}
\newtheorem{remark}[theorem]{Remark}

\newtheorem*{acknow}{Acknowledgement}
\numberwithin{equation}{section}

\numberwithin{mytheorem}{subsection}
\def \no {\noindent}
\title{Results on the topology of generalized real Bott manifolds}
\author{Raisa DSouza}
\author{V. Uma}
\address{Department of Mathematics, IIT Madras, Chennai, India}
\email{raisadsouza1989@gmail.com; vuma@iitm.ac.in}
\date{}

\begin{document}

\begin{abstract}
 Generalized Bott manifolds (over $\mathbb C$ and $\mathbb R$) have been defined by Choi, Masuda and Suh in \cite{qtmoverproduct}. In this article we extend the results of \cite{umaraisa} on the topology of real Bott manifolds to generalized real Bott manifolds. We give a presentation of the fundamental group, prove that it is solvable and give a characterization for it to be abelian. We further prove that these manifolds are aspherical only in the case of real Bott manifolds and compute the higher homotopy groups. Furthermore, using the presentation of the cohomology ring with $\mathbb Z_2$-coefficients, we derive a combinatorial characterization for orientablity and spin structure.
\end{abstract}

\footnotetext{{\bf 2010 Subject Classification:} Primary: 55R99, Secondary: 57S25}

\maketitle

\section{Introduction}
A  generalized Bott tower has been defined by Choi, Masuda and Suh in \cite{masudachoigeneralized, qtmoverproduct} as a sequence of projective bundles 
 \begin{equation}\label{bt}
 B_k\xrightarrow{p_k}B_{k-1}\xrightarrow{p_{k-1}}\cdots\xrightarrow{p_2}B_1\xrightarrow{p_1}B_0=\{\,\text{a point}\,\}
 \end{equation}
                                                                                                                                                                                                                                  where $B_i$ for $i=1,\cdots,k$ is the projectivization of the Whitney sum of $n_i+1$ $\mathbb F$-line bundles over $B_{i-1}$, ($\mathbb F=\mathbb C$ or $\mathbb R$). This generalizes the construction of Bott towers by Grossberg and Karshon \cite{grossbergkarshon} in the case when $\mathbb F=\mathbb C$ and $n_i=1$ for every $i$. The $i$th stage $B_i$ of the tower can be alternately realized as a quasitoric manifold (when $\mathbb F=\mathbb C$) or a small cover (when $\mathbb F=\mathbb R$) over $\prod_{i=1}^k\Delta^{n_i}$ where $\Delta^{n_i}$ is the $n_i$-simplex. We call each $B_i$ a generalized Bott manifold (over $\mathbb F$). In the case when $n_i=1$ for every $i$ we call each $B_i$ a Bott manifold. \par                                                                                                                                                                                                                              
                                                                                                                                                                                                                                  In \cite{masudachoigeneralized}, Choi, Masuda and Suh study these manifolds from the viewpoint of the \emph{cohomological rigidity problem for toric manifolds}. Also in \cite[Theorem 6.4]{qtmoverproduct}, they characterize generalized Bott manifolds among the quasitoric manifolds over a product of simplices, as those which admit an invariant almost complex structure under the action of $(S^1)^k$. Moreover, in \cite[Section 7]{qtmoverproduct} they describe the cohomology ring of any quasitoric manifold over a product of simplices and in \cite[Section 8]{qtmoverproduct} they give a sufficient condition in terms of the cohomology ring for quasitoric manifolds over a product of simplices to be diffeomorphic to a generalized Bott manifold. \par
 
In this article our main objects of study are $B_k$ (when $\mathbb F$ = $\mathbb R$) which we call the generalized \emph{real} Bott manifolds. We begin Section \ref{2} with the definition of a \emph{small cover} \cite{davisjanus} over a product of simplices paralleling that of a quasitoric manifold as in \cite{qtmoverproduct}. It follows from \cite[Section 4.1 and 4.2]{kurokiprojbundles} that every generalized real Bott manifold is a small cover over a product of simplices. In Proposition \ref{scisgbt} we prove the converse, that any small cover over a product of simplices is a generalized real Bott manifold which is the small cover analogue of \cite[Proposition 6.2]{qtmoverproduct}. \par 

In Section \ref{3} we study the fundamental group of $B_k$. More precisely, in Theorem \ref{fungp}, we give a presentation of $\pi_1(B_k)$ which generalizes the description of the fundamental group of a real Bott manifold as given in \cite[Lemma 3.2]{masudakamishimacohorigid}, \cite[Lemma 3.2]{leebumiterated} and \cite[Theorem 2.1]{umaraisa}. We further derive a characterization of the generalized real Bott manifolds which have abelian fundamental group. Furthermore, we show that $B_k$ is aspherical if and only if it is a real Bott manifold. We study further group theoretic properties of $\pi_1(B_k)$ in Propositions \ref{sol} and \ref{nil}.\par

We begin Section \ref{4} by recalling the description of the cohomology ring with $\mathbb Z_2$-coefficients of the generalized real Bott manifolds. More precisely, in Proposition  \ref{cohoring}, we give a presentation with generators and relations of $H^*(B_k;\mathbb Z_2)$. Further, we give a combinatorial criterion for $B_k$ to be orientable in Theorem \ref{orientable}  and for $B_k$ to be spin in Theorem \ref{spin}, by deriving closed formulae for the first and second Stiefel-Whitney classes of $B_k$ (see Theorem \ref{totswc}, Corollary \ref{w1result} and Corollary \ref{w2result}). These generalize the orientablity criterion given by Kamishima and Masuda for real Bott manifolds in \cite[Lemma 2.2]{masudakamishimacohorigid} and the criterion for spin structure for real Bott manifolds given by Ga\c{s}ior in \cite{gasiorspin} and the authors of the present paper in \cite{umaraisaspin}.

\section{Small cover over a product of simplices}\label{2}
In this section we define small cover over a product of simplices and recall some preliminary results.\par
Let $\Delta^{n_i}$ be the standard $n_i$-simplex for $1\leq i\leq k$ and let $P=\prod_{i=1}^k\Delta^{n_i}$. Then $P$ is a simple polytope of dimension $n=\sum_{i=1}^k n_i$. 

We know that an $m$-simplex has $m+1$ facets. Let $f^i_{l_i}$ for $0\leq l_i\leq n_i$ denote the facets of $\Delta^{n_i}$ for $1\leq i\leq k$. Then the facets of $P$ are as follows :
\begin{equation}
  F^i_{l_i}=\Delta^{n_1}\times\cdots\times\Delta^{n_{i-1}}\times f^i_{l_i}\times\Delta^{n_{i+1}}\times\cdots\times\Delta^{n_k}
\end{equation}
for $0\leq l_i\leq n_i$ and $1\leq i\leq k$. Let $\mathcal F=\{F^i_{l_i}:1\leq i\leq k\,,\, 1\leq l_i\leq n_i\}$

\begin{remark}\label{facetsofP}
 Note that $ F^i_{l_i}\cap F^i_{l_i^\prime}=\Delta^{n_1}\times\cdots\times (f^i_{l_i}\cap f^i_{l_i^\prime})\times\cdots\times\Delta^{n_k}$. If $n_i\geq2$ then $f^i_{l_i}\cap f^i_{l_i'}\neq\emptyset$ for $1\leq l_i,l_i'\leq n_i$. If $n_i=1$, $f^i_{l_i}\cap f^i_{l_i'}\neq\emptyset$ if and only if $l_i=l_i'$. Also for $i\neq j$, we have $ F^i_{l_i}\cap  F^j_{l_j^\prime}=\Delta^{n_1}\times\cdots\times f^i_{l_i}\times\cdots\times f^j_{l_j^\prime}\times\cdots\times\Delta^{n_k}$.\\ 
 Thus, $F^i_{l_i}\cap F^j_{l_j'}=\emptyset$ if and only if $i=j$, $n_i=1$ and $l_i\neq l_i'$.
\end{remark}

Similarly, we know that an $m$-simplex has $m+1$ vertices. Let $v^i_{l_i}$ for $0\leq l_i\leq n_i$ denote the vertices of $\Delta^{n_i}$ for $1\leq i\leq k$. Then the vertex set of $P$ is as follows :
\begin{equation}
 \mathcal V=\{v_{l_1\cdots l_k} = v^1_{l_1}\times\cdots\times v^k_{l_k}: 0\leq l_i\leq n_i\,,\,1\leq i\leq k\}
\end{equation}

We shall assume that $v^i_{l_i}$ is the vertex of $\Delta^{n_i}$ opposite the facet $f^i_{l_i}$. Then it is easy to see that $v_{l_1\cdots,l_k}$ is the vertex at which the facets in $\mathcal F - \{F^1_{l_1},\cdots, F^k_{l_k}\}$ intersect.

\begin{definition}\label{scdefn}
 Let $\lambda:\mathcal F\rightarrow \mathbb Z_2^n$ be defined as follows,
 \begin{equation}
 \begin{split}
 ~& \lambda( F_1^1)=\mathbf e_1,\cdots,\lambda (F^1_{n_1})=\mathbf e_{n_1},\cdots, \lambda( F^k_{n_k})=\mathbf e_n. \\
 ~& \lambda( F^1_0)=\mathbf a_1\,,\,\lambda( F^2_0)=\mathbf a_2,\cdots, \lambda( F^k_0)=\mathbf a_k.
 \end{split}
\end{equation}
where $\mathbf e_i=(0,\cdots,1,\cdots,0)$ (with $1$ in the $i$th position) for $1\leq i\leq n$ and $\mathbf a_i=(\mathbf a_i^1,\cdots,\mathbf a_i^k)\in\mathbb Z_2^n$, for $1\leq i\leq k$ and each $\mathbf a_i^j=(a_{i,1}^j,a_{i,2}^j,\cdots,a_{i,n_j}^j)\in\mathbb Z_2^{n_j}$. We assume that whenever a set of facets of $P$ intersect along a face their images under $\lambda$ form a part of a basis of $\mathbb Z_2^n$. The small cover over $P$ associated to the \emph{characteristic function} $\lambda$ is defined as $M(\lambda)= \mathbb Z^n_2 \times P / \sim$ where $(t, p) \sim (t' , p')$ if and only if $p = p'$ and $t \cdot (t')^{-1}\in G_{F(p)}$ . Here $F(p)$ is the unique face of $P$ which contains $p$ in its relative interior and $G_{F (p)}$ is the rank-$l$ subgroup of $\mathbb Z_2^n$ spanned by the images under $\lambda$ of the $l$ facets of $P$ that intersect at $F(p)$ (see \cite[Section 1.5]{davisjanus}).
\end{definition}

The function $\lambda$ determines the following $k\times n$ matrix $A$, which can be viewed as a $k\times k$ vector matrix
\begin{equation}\label{matrix}
A=
 \begin{pmatrix}
  a_{1,1}^1 & \cdots & a_{1,n_1}^1 & \cdots & a_{1,1}^k & \cdots & a_{1,n_k}^k\\
  \vdots & ~ & \vdots & ~ & \vdots & ~ & \vdots \\
  a_{k,1}^1 & \cdots & a_{k,n_1}^1 & \cdots & a_{k,1}^k & \cdots & a_{k,n_k}^k  
 \end{pmatrix}
 =\begin{pmatrix}
   \mathbf{a}_1^1 & \cdots & \mathbf{a}_1^k\\
   \vdots & \ddots & \vdots \\
   \mathbf{a}_k^1 & \cdots & \mathbf{a}_k^k
  \end{pmatrix}
\end{equation}

Thus $M(\lambda)$ can also be thought of a small cover associated to the matrix $A$ and we sometimes denote it by $M(A)$.\par
\no
\begin{remark}\label{QC}
{\it Quotient construction.} A small cover over a simple convex polytope $P$ of dimension $n$ with $m$ facets can be realized as the quotient of the \emph{moment-angle manifold} of the polytope by certain free action of the compact real torus of dimension $m-n$. It can be seen that for the polytope $P=\prod_{i=1}^{k}\Delta^{n_i}$ the moment-angle manifold is $X=\prod_{i=1}^k S^{\,n_i}$ (see \cite[Section 1.10 and 4.1]{davisjanus}). In particular, it follows that $M(A)$ is the quotient of $X$ by the action of $G=\mathbb Z_2^k$ given by
\begin{equation}\label{action}
\begin{split}
 ~ & (g_1,\cdots,g_k)\cdot\big((x_0^1,\cdots,x_{n_1}^1),\cdots,(x_0^k,\cdots,x_{n_k}^k)\big)\\
 ~ & = \big((g_1\cdot x_0^1,(g_1^{a_{1,1}^1}\cdots g_k^{a_{k,1}^1})\cdot x_1^1,\cdots,(g_1^{a_{1,n_1}^1}\cdots g_k^{a_{k,n_1}^1})\cdot x_{n_1}^1),\cdots,\\
 ~ & \qquad \qquad (g_k\cdot x_0^k,(g_1^{a_{1,1}^k}\cdots g_k^{a_{k,1}^k})\cdot x_1^k,\cdots,(g_1^{a_{1,n_k}^k}\cdots g_k^{a_{k,n_k}^k})\cdot x_{n_k}^k)\big)
 \end{split}
\end{equation}
where $(g_1,\cdots,g_k)\in G$ and $(x_0^i,\cdots,x_{n_i}^i)\in S^{\,n_i}$ for $i=1,\cdots k$ (the proof is analogous to \cite[Proposition 4.3]{qtmoverproduct}).\par
\end{remark}

For $1\leq j_1\leq n_1\,,\,\cdots\,,\,1\leq j_k\leq n_k$ let $A_{j_1\cdots j_k}$ be the following $k\times k$ submatrix of $A$ :
\begin{equation}\label{submatrix}
A_{j_1\cdots j_k}=
 \begin{pmatrix}
  a_{1,j_1}^1 & \cdots & a_{1,j_k}^k\\
  \vdots & ~ & \vdots\\
  a_{k,j_1}^1 & \cdots & a_{k,j_k}^k
 \end{pmatrix}
\end{equation}

The following Proposition is the $\mathbb Z_2$ analogue of \cite[Lemma 3.2]{qtmoverproduct}. 

\begin{proposition}\label{minorbasisequivalence}
 Let $A$ be a matrix of the form (\ref{matrix}) that is associated to the characteristic function $\lambda$ of a small cover over $P$. The following two properties are equivalent,
 \begin{enumerate}
 \item The submatrices $A_{j_1\cdots j_k}$ for  $1\leq j_1\leq n_1\,,\,\cdots\,,\,1\leq j_k\leq n_k$ have the property that all their principal minors are $1$.
 \item If any $n$ facets of $P$ intersect at a vertex then the image of the corresponding facets under $\lambda$ forms a basis of $\mathbb Z_2^n$.
 \end{enumerate}
\end{proposition}

\begin{proof}
 Let $\delta=\begin{vmatrix}
  a_{i_1,j_{i_1}}^{i_1} & \cdots & a_{i_1,j_{i_r}}^{i_r}\\
  \vdots & ~ & \vdots\\
  a_{i_r,j_{i_1}}^{i_1} & \cdots & a_{i_r,j_{i_r}}^{i_r}
 \end{vmatrix}$ be a principal minor of $A_{j_1\cdots j_k}$. Let $I$ be the $n\times n$ identity matrix. Let us identify the columns of $I$ with $\lambda(F^1_1),\cdots,\lambda(F^1_{n_1}),\cdots,\lambda(F^k_1),\cdots,\lambda(F^k_{n_k})$ in that order. Let $A'$ be the matrix obtained from $I$ by replacing the columns corresponding to $\lambda(F^{i_1}_{j_{i_1}}),\cdots,\lambda(F^{i_r}_{j_{i_r}})$ by the $i_1^{\text{ th}}$, $\cdots, i_r^{\text{ th}}$ rows of the matrix $A$ in (\ref{matrix}). It is easy to see that $\delta=\text{det }A'$. Now each column of $A'$ is the image of some facet of $P$ under $\lambda$ and these facets intersect at the vertex $v_{J}$ where $J=0\cdots 0\,j_{i_1}\,0\cdots 0\,j_{i_2}\,0\cdots0\,j_{i_r}\,0\cdots 0$. Further, we know that $\text{det }A'\equiv1\bmod2$ if and only if the columns of $A'$ form a basis of $\mathbb Z_2^n$. Hence the proposition.
\end{proof}

\begin{proposition}\label{conjugate}
 Let $A$ be a $k\times k$ vector matrix as in (\ref{matrix}) associated to the characteristic function of a small cover over a product of simplices . Then $A$ is conjugate by a permutation matrix to a unipotent upper triangular vector matrix of the following form :
 \begin{equation}\label{UTmatrix}
  \begin{pmatrix}
   \mathbf 1 & \mathbf b_1^2 & \mathbf b_1^3 & \cdots & \mathbf b_1^k\\
  \mathbf 0 & \mathbf 1 & \mathbf b_2^3 & \cdots & \mathbf b_2^k\\
  \vdots & ~ & ~ & ~ & \vdots \\
  \mathbf 0 & \cdots & \cdots & \mathbf 1 & \mathbf b_{k-1}^k\\
  \mathbf 0 & \cdots & \cdots & \mathbf 0 & \mathbf 1
  \end{pmatrix}
 \end{equation}
 where $\mathbf 0=(0,\cdots,0),\,\mathbf 1=(1,\cdots,1)$ of appropriate sizes and $\mathbf b_i^j=(b_{i,1}^j,\cdots,b_{i,n_j}^j)\in\mathbb Z_2^{n_j}$. Here multiplication of matrix entries is scalar multiplication and addition is vector addition.
\end{proposition}

\begin{proof}
 Since the entries of $A$ are all from $\mathbb Z_2$ it follows from Proposition \ref{minorbasisequivalence} that all the principal minors of the submatrices $A_{j_1\cdots j_k}$ for  $1\leq j_1\leq n_1\,,\,\cdots\,,\,1\leq j_k\leq n_k$ are 1. The result then follows from \cite[Lemma 5.1]{qtmoverproduct}.
\end{proof}

\begin{remark}\label{J}
Note that we can view $P$ as $\prod_{i=1}^k\Delta^{n_{\sigma(i)}}$ where $\sigma\in S_k$. Then we obtain a matrix $\tilde A=E_\sigma AE_\sigma^{-1}$ where $E_\sigma$ is the permutation matrix corresponding to $\sigma$. It is easy to see that the small covers associated to $A$ and $\tilde A$ are equivariantly diffeomorphic (the proof is similar to that of \cite[Proposition 5.1]{masudachoidigraph} also see \cite[Section 5]{qtmoverproduct}). In view of this fact and Proposition \ref{conjugate} we will henceforth assume that the matrix $A$ has the form (\ref{UTmatrix}).
\end{remark}

\begin{proposition}\label{scisgbt}
Every small cover over a product of simplices is a generalized real Bott manifold.
\end{proposition}

\begin{proof}
 The proof is the small cover analogue of the proof of \cite[Proposition 6.2]{qtmoverproduct}.  Let $M(A)$ be the small cover over $\prod_{i=1}^k\Delta^{n_i}$ associated to the matrix $A$. We may assume by Proposition \ref{conjugate} that $A$ is of the form (\ref{UTmatrix}). Let $X_j=\prod_{i=1}^jS^{\,n_i}$, which is the moment-angle manifold of $\prod_{i=1}^j\Delta^{n_i}$ for $j=1,\cdots,k$. The group $G_k=\mathbb Z_2^k$ acts on $X_k$ as in (\ref{action}) and $X_k/G_k=M(A)$. Let $B_j=X_j/G_k$ where $G_k$ acts on $X_j$ as follows 
 \begin{equation*}
\begin{split}
 ~ & (g_1,\cdots,g_k)\cdot\big((x_0^1,\cdots,x_{n_1}^1),\cdots,(x_0^j,\cdots,x_{n_j}^j)\big)\\
 ~ & = \big((g_1\cdot x_0^1,(g_1^{a_{1,1}^1}\cdots g_k^{a_{k,1}^1})\cdot x_1^1,\cdots,(g_1^{a_{1,n_1}^1}\cdots g_k^{a_{k,n_1}^1})\cdot x_{n_1}^1),\cdots,\\
 ~ & \qquad \qquad (g_j\cdot x_0^j,(g_1^{a_{1,1}^j}\cdots g_k^{a_{k,1}^j})\cdot x_1^j,\cdots,(g_1^{a_{1,n_j}^j}\cdots g_k^{a_{k,n_j}^j})\cdot x_{n_j}^j)\big)
 \end{split}
\end{equation*}
 We get a sequence $M(A)=B_k\xrightarrow{p_k}B_{k-1}\rightarrow\cdots B_1\xrightarrow{p_1}B_0=\{\text{pt}\}$ induced by the natural projections of $X_j$ on $X_{j-1}$ for $j=1,\cdots,k$. 
 
 Since $A$ is of the form (\ref{UTmatrix}), the last $k-j$ factors of $G_k$ act trivially on $X_j$ so that the $G_k$ action reduces to a $G_j$ action on $X_j$ with $X_j/G_k\simeq X_j/G_j$. Further, the last factor of $G_j$ acts on the last factor of $X_j$ as scalar multiplication and trivially on other factors. It follows that $p_j:B_j=X_j/G_j\rightarrow B_{j-1}=X_{j-1}/G_{j-1}$ is the projectivization of a real vector bundle $\xi_j$ over $B_{j-1}$ associated to the principal $G_{j-1}$-bundle $X_{j-1}\rightarrow X_{j-1}/G_{j-1}=B_{j-1}$. In fact, $\xi_j=(X_{j-1}\times V_j)/G_{j-1}$ where $V_j=\mathbb R^{n_j+1}$ with the action of $G_{j-1}$ given by 
\begin{equation*}
\begin{split}
 ~ & (g_1,\cdots,g_{j-1})\cdot(x_0^j,\cdots,x_{n_j}^j)\\
 ~ & \qquad= \big(x_0^j,(g_1^{a_{1,1}^j}\cdots g_{j-1}^{a_{j-1,1}^j})\cdot x_1^j,\cdots,(g_1^{a_{1,n_j}^j}\cdots g_{j-1}^{a_{j-1,n_j}^j})\cdot x_{n_j}^j\big).
 \end{split}
\end{equation*}
\end{proof}

Let $N=\mathbb Z^n$ be the lattice with basis $\{u_1^1,\cdots,u_{n_1}^1,\cdots,u_1^k,\cdots,u_{n_k}^k\}$.
Define $u^i_0=-(b_{i,1}^1\cdot u_1^1+\cdots+b_{i,n_1}^1\cdot u^1_{n_1}+\cdots+b_{i,1}^k\cdot u_1^k+\cdots+b_{i,n_k}^k\cdot u_{n_k}^k)$ for $1\leq i\leq k$ where $b_{i,l_i}^j$ are entries from the matrix (\ref{UTmatrix}). Let $\Sigma_k$ be the fan in $N$ consisting of cones generated by any sub-collection of $\{u^1_0,\cdots,u^1_{n_1},\cdots,u^k_0,\cdots,u^k_{n_k}\}$ that doesn't contain $\{u^i_0,\cdots,u^i_{n_i}\}$ for $1\leq i\leq k$.

\begin{proposition}
 The generalized real Bott manifold $B_k$ is the real toric variety associated to the smooth projective fan $\Sigma_k$. 
\end{proposition}

\begin{proof}
 We proceed by induction on $k$. When $k=1$ we have $B_1=\mathbb{RP}^n_1$ which is the real toric variety associated to the fan $\Sigma_1$. Let $N'=\mathbb Z^{n-n_k}$ be the lattice with basis $\{(u_1^1)',\cdots,(u_{n_1}^1)',\cdots,(u_1^{k-1})',\cdots,(u_{n_{k-1}}^{k-1})'\}$. For $1\leq i\leq k-1$ let $$(u^i_0)'=-\big(b_{i,1}^1\cdot (u_1^1)'+\cdots+b_{i,n_{k-1}}^{k-1}\cdot (u^{k-1}_{n_{k-1}})'\big).$$ Consider the fan $\Sigma'$ in $N'$ consisting of cones generated by any sub-collection of \\ $\{(u_0^1)',\cdots,(u_{n_1}^1)',\cdots,(u_{0}^{k-1})',\cdots,(u_{n_{k-1}}^{k-1})'\}$ that doesn't contain $\{(u_0^i)',\cdots,(u_{n_i}^i)'\}$ for all $1\leq i\leq k-1$. Then $\Sigma'$ is isomorphic to the fan $\Sigma_{k-1}$. By the induction hypothesis $B_{k-1}$ is the real toric variety associated to the smooth projective fan $\Sigma_{k-1}\simeq\Sigma'$. Now $B_k=\mathbb P(\mathbf 1\oplus L_1\oplus\cdots\oplus L_{n_k})$ where  $L_1,\cdots,L_{n_k}$ are real line bundles over $B_{k-1}$. By Lemma \ref{ELB} each $L_i$ for $1\leq i\leq n_k$, is isomorphic to a $T$-equivariant real algebraic line bundle $L_i'$ over $B_{k-1}$ where $T=(\mathbb R^*)^{n-n_k}$. Hence $B_k$ is homeomorphic to $\mathbb P(\mathbf 1\oplus L_1'\oplus\cdots\oplus L_{n_k}')$. Thus without loss of generality we can assume that $L_i$ for $1\leq i\leq n_k$ is a $T$-equivariant real algebraic line bundle over $B_{k-1}$. Let $h_1,\cdots,h_{n_k}$ be the \emph{support functions} corresponding to $L_1,\cdots,L_{n_k}$ respectively. In view of Remark \ref{ELBrem}, 
 \begin{equation}\label{hj} h_j((u^i_{l_i})')= 0\qquad \text{for}\quad1\leq i\leq k-1\,,\,1\leq l_i\leq n_i \end{equation}
 and let \begin{equation}b_{i,j}^k:=h_j((u^i_0)')\qquad\text{for}\quad1\leq i\leq k-1.\end{equation}

 Let $N''=\mathbb Z^{n_k}$ with basis $\{(u_1^k)'',\cdots,(u_{n_k}^k)''\}$. Let $(u_0^k)'':=-((u_1^k)''+\cdots+(u_{n_k}^k)'')$. Let $\Sigma''$ denote the fan in $N''$ consisting of cones generated by any proper subset of $\{(u_0^k)'',\cdots,(u_{n_k}^k)''\}$. We get an exact sequence of fans 
 \begin{equation}\label{esf}
  0\rightarrow (N'',\Sigma'')\xrightarrow{\phi''}(N,\Sigma_k)\xrightarrow{\phi'}(N',\Sigma')\rightarrow0
 \end{equation}
 where, $\phi''((u_{l_k}^k)'')=u_{l_k}^k$ for $1\leq l_k\leq n_k$ and $\phi'(u^i_{l_i})=(u^i_{l_i})'$ for $1\leq i\leq k-1\,,\,1\leq l_i\leq n_i$ and $\phi'(u^k_{l_k})=0$ for $1\leq l_k\leq n_k$.
 
 Define $\psi:N'\simeq\mathbb Z_2^{n-n_k}\rightarrow \mathbb Z_2^n\simeq N$ by $\psi(y)=\big(y,-h_1(y),\cdots,-h_{n_k}(y)\big)$. Then by (\ref{hj}) $\psi$ gives a splitting of (\ref{esf}). That is $N=N'\oplus N''$. Further, if we let $\tilde{\sigma}$ denote the image of any $\sigma\in\Sigma'$ under $\psi$ then,
 $$\tilde{\Sigma}=\{\tilde{\sigma}:\sigma\in\Sigma'\}$$
 It can be seen that $$\Sigma_k=\{\tilde{\sigma}+\sigma'':\sigma\in\Sigma'\text{ and } \sigma''\in\Sigma''\}$$
 Then from the real analogue of \cite[Proposition 1.33]{odabook} we have that $B_k$ is the real toric variety corresponding to the smooth projective fan $\Sigma_k$.
 
 It also follows independently from \cite[Section 4.1 and Section 4.2]{kurokiprojbundles}, that $B_k$ is a small cover over the product of simplices $\prod_{i=1}^k\Delta^{n_i}$.
\end{proof}

\begin{lemma}\label{ELB}
 Let $X_\Sigma$ be the real toric variety with dense torus $(\mathbb R^*)^n$, associated to the smooth projective fan $\Sigma$ in the lattice $N\simeq\mathbb Z^n$. Every real line bundle over $X_\Sigma$ is isomorphic to an $(\mathbb R^*)^n$-equivariant algebraic real line bundle.
\end{lemma}

\begin{proof}
 We know that $X_\Sigma=\bigcup_{\sigma\in\Sigma}U_\sigma$ where $U_\sigma=\operatorname{Hom}_{sg}(M\cap\sigma^\vee,\mathbb R)$ , $\sigma^\vee$ is the dual cone of $\sigma$ in the dual lattice $M$. Let $\{\rho_j:1\leq j\leq d\}$ be the edge vectors in $\Sigma$ and let $v_j$ be the primitive vector generating the edge $\rho_j$ for $1\leq j\leq d$. Let $\mathscr L_j$ be the algebraic real line bundle corresponding to the piecewise linear data $\{m_\sigma\}_{\sigma\in\Sigma}$ (see \cite[Proposition 2.1]{odabook}) where $m_\sigma\in M$ is defined as $m_\sigma=0$ if $v_j$ is not on an edge of $\sigma$ and if $v_j$ is on an edge of $\sigma$ then $-m_\sigma$ is the element in the dual basis to the basis of $N$ consisting of generators of $\sigma$ such that $\langle- m_\sigma,v_j\rangle=1$. Since $m_\sigma-m_{\sigma'}\in (\sigma\cap\sigma')^\perp$, we can define as in \cite[p. 69]{odabook}, the real $(\mathbb R^*)^n$-equivariant line bundle associated to $\{m_\sigma\}$ as follows : 
 $$\mathscr L_j=\bigcup_{\sigma\in\Sigma}U_\sigma\times\mathbb R/\sim \text{ where } (x,r)\sim(x,x(m_\sigma-m_{\sigma'})\cdot r) \text{ for } x\in U_\sigma\cap U_{\sigma'}$$ 
 Then $\{\mathscr L_j:1\leq j\leq d\}$ are the canonical real line bundles over $X_\Sigma$. Each $\mathscr L_j$ has an $(\mathbb R^*)^n$-equivariant sections $s_j$ such that the zero locus of $s_j$ is the $(\mathbb R^*)^n$-invariant co-dimension $1$ real subvariety $V(\rho_j)$ of $X_\Sigma$ (see \cite[Proposition 2.1]{odabook} and \cite[Example 2.10.5]{transgroupsmukherjee}).  Here $V(\rho_j)$ is the closure of the orbit ${O_{\rho_j}}=\operatorname{Hom}_{gp}(M\cap\rho_j^\perp,\mathbb R^*)$. Then, $w_1(\mathscr L_j)=[V(\rho_j)]\in H^1(X_\Sigma;\mathbb Z_2)$, where $[V(\rho_j)]$ is the Poincar\'{e} dual to the homology class of $V(\rho_j)$. Now by \cite[Theorem 4.3.1]{jurkiewiczthesis}, as a $\mathbb Z_2$-vector space, $H^1(X_\Sigma;\mathbb Z_2)$ is generated by $[V(\rho_j)]$ for $1\leq j\leq d$. Let $L$ be a real line bundle over $X_\Sigma$. Since $w_1(L)\in H^1(X_\Sigma;\mathbb Z_2)$ we have $w_1(L)=\sum_{j=1}^d c_j\cdot[V(\rho_j)]=w_1(\otimes \mathscr L_j^{c_j})$. Thus $L\simeq \otimes L_j^{c_j}$ (\cite[p. 250]{husemollerfibre}).
\end{proof}

\begin{remark}\label{ELBrem}
 Since $\Sigma$ is a smooth complete fan, we may assume without loss of generality that $\rho_1,\cdots,\rho_n$ span a cone in $\Sigma$. Then $\{v_1,\cdots,v_n\}$ is a basis of $N$ with dual basis $\{v^*_1,\cdots,v_n^*\}$ of $M$. From \cite[Theorem 4.3.1]{jurkiewiczthesis}, for $1\leq i\leq n$ we have the linear relations $\sum_{j=1}^d\langle v_i^*,v_j\rangle\cdot[V(\rho_j)]$ so that $H^1(X_\Sigma;\mathbb Z_2)=\langle[V(\rho_j)]:n+1\leq j\leq d\rangle\simeq \mathbb Z_2^{d-n}$. Therefore if $L$ is any real line bundle over $X_\Sigma$ then $L\simeq\bigotimes_{i=n+1}^d\mathscr L_i^{c_i}=:\mathscr L_h$. Here $\mathscr L_h$ is the line bundle corresponding to the support function given by $h(v_i)=-c_i$ for $n+1\leq i\leq d$ and $h(v_i)=0$ otherwise.
\end{remark}

\section{The fundamental group}\label{3}
We begin this section by giving a presentation of $\pi_1(B_k)$ with generators and relations. This is obtained by applying \cite[Proposition 3.1]{umafundamental}. Using the presentation we characterize those $B_k$ with abelian fundamental group. Again, using \cite[Theorem 6.1]{umafundamental} we show that $B_k$ is aspherical only when it is a real Bott manifold. In Section \ref{FGTP}, we derive further group theoretic properties of $\pi_1(B_k)$. In particular we prove that $\pi_1(B_k)$ is solvable and that $\pi_1(B_k)$ is nilpotent if and only if it is abelian.

\subsection{Presentation of $\pi_1(B_k)$}
Recall that we have an exact sequence 
\begin{equation}
 1\rightarrow\pi_1(B_k)\rightarrow W(\Sigma_k)\rightarrow\mathbb Z_2^n\rightarrow 1
\end{equation}
where $W(\Sigma_k)$ is the right angled Coxeter group associated to $\Sigma_k$ with the following presentation :
\begin{equation}\label{W}
\begin{split}
 W=W (\Sigma_k)& =\langle s_{i,l_i} : 1\leq i\leq k\,,\,0\leq l_i\leq n_i\,|\, s_{i,l_i}^2=1\,,\,(s_{i,l_i}\cdot s_{j,l_j^\prime})^2=1\text{ for }i\neq j\\
 ~ & \qquad\text{ and } (s_{i,l_i}\cdot s_{i,l_i^\prime})^2=1\ \forall\ i \text{ such that } n_i\geq2 \rangle.
 \end{split}
\end{equation}
The last arrow in the above exact sequence is obtained by composing the natural abelianization map from $W$ to $\mathbb Z_2^{n+k}$ with the characteristic map $\lambda$ from $\mathbb Z_2^{n+k}$ to $\mathbb Z_2^n$. (Here $\lambda$ is extended by linearity from Definition \ref{scdefn})

For each $1\leq j\leq k$ we define the reduced word 
\begin{equation}
 \alpha_j:=s_{j,0}\cdot s_{1,1}^{a_{j,1}^1}\cdot s_{1,2}^{a_{j,2}^1}\cdots s_{1,n_1}^{a_{j,n_1}^1}\cdots s_{k,1}^{a_{j,1}^k}\cdots s_{k,n_k}^{a_{j,n_k}^k}
\end{equation}
in $W$.

For every $\varepsilon=(\varepsilon_{1,1},\cdots,\varepsilon_{1,n_1},\cdots,\varepsilon_{k,1},\cdots,\varepsilon_{k,n_k})\in\mathbb Z_2^n$ we define the following :
\begin{enumerate}
 \item $t_\varepsilon=s_{1,1}^{\varepsilon_{1,1}}\cdots s_{1,n_1}^{\varepsilon_{1,n_1}}\cdots s_{k,1}^{\varepsilon_{k,1}}\cdots s_{k,n_k}^{\varepsilon_{k,n_k}}$.
 \item $B^p=(B_{1,1}^p,\cdots,B_{1,n_1}^p,\cdots,B_{k,1}^p,\cdots,B_{k,n_k}^p)$ where $B_{i,l_i}^p=\varepsilon_{i,l_i}+a_{p,l_i}^i$ for $1\leq i,p\leq k$ and $1\leq l_i\leq n_i$.
 \item $C^{p,q}=(C_{1,1}^{p,q},\cdots,C_{1,n_1}^{p,q},\cdots, C_{k,1}^{p,q},\cdots,C_{k,n_k}^{p,q})$ where $C_{i,l_i}^{p,q}=\varepsilon_{i,l_i}+a_{p,l_i}^i + a_{q,l_i}^i$ for $1\leq i,p,q\leq k$ and $1\leq l_i\leq n_i$.
\end{enumerate}
From the relations in $W$ we get, 
\begin{equation}\label{yepsilon}
 t_{\varepsilon}\cdot\alpha_j\cdot t_\varepsilon=\left\{\begin{array}{ccl}
                                                  \alpha_j & , & \varepsilon_{j,1}=0\text{ or }n_j\geq2\\
                                                  \alpha_j^{-1} & , & \text{otherwise} 
                                                 \end{array}\right.
\end{equation}
This can be seen as follows : if $n_j=2$  then $s_{j,0}$ commutes with all $s_{i,l_i}$ so that $t_{\varepsilon}\cdot\alpha_j\cdot t_\varepsilon=\alpha_j$. When $n_j=1$ we get
\begin{equation*}
 t_{\varepsilon}\cdot\alpha_j\cdot t_\varepsilon=s_{j,1}^{\varepsilon_{j,1}}\cdot s_{j,0}\cdot s_{1,1}^{a_{j,1}^1}\cdots s_{1,n_1}^{a_{j,n_1}^1}\cdots s_{j,1}^{a_{j,1}^j}\cdots s_{k,n_k}^{a_{j,n_k}^k}\cdot s_{j,1}^{\varepsilon_{j,1}}=\left\{\begin{array}{ccl}
                                                                                    \alpha_j & , & \varepsilon_{j,1}=0\\
                                                                                    \alpha_j^{-1} & , & \varepsilon_{j,1}=1 
                                                                                    \end{array}\right.
\end{equation*}
where $a_{j,1}^j=1$ . Note that $\alpha_j^2=1$ when $n_j\geq2$.

\begin{remark}\label{l}
 In view of Remark \ref{J} we may assume that there exists an $l\leq k$ such that $n_1,\cdots,n_l\geq2$ and $n_{l+1}=\cdots=n_k=1$. Let $M^l$ be the real part of the toric variety associated to the fan $\Sigma_l$. Note that if $l=0$ then $B_k$ is a real Bott manifold.
\end{remark}

\begin{theorem}\label{fungp}
 The fundamental group of the generalized real Bott manifold $B_k$ has a presentation $\pi_1(B_k)=\langle\,S\,|\,R\,\rangle$ with generators,
 \begin{equation}\label{gtrs}
  S=\{\alpha_j:1\leq j\leq k\}
 \end{equation}
 and relations,
 \begin{equation}\label{rels}
  \begin{split}
  R & =\{(\alpha_p\,\alpha_q)^2 : 1\leq p<q\leq l\}\bigcup\{\alpha_p^2:1\leq p\leq l\}\\
  ~ & \qquad\bigcup\{ x_{p,q} : l+1\leq p<q\leq k\}\bigcup\{x_{p,q}^\prime:1\leq p\leq l<q\leq k\}
  \end{split}
 \end{equation}
 where, \begin{equation*}
         x_{p,q}=\left\{\begin{array}{ccl}
                         \alpha_p\,\alpha_q\,\alpha_p^{-1}\,\alpha_q^{-1} & , & a_{p,1}^q=0\\
                         \alpha_p\,\alpha_q^{-1}\,\alpha_p^{-1}\,\alpha_q^{-1} & , & a_{p.1}^q=1
                        \end{array}\right.
        \end{equation*}
 and, \begin{equation*}
         x_{p,q}^\prime=\left\{\begin{array}{lcl}
                                     \alpha_p\,\alpha_q\,\alpha_p\,\alpha_q^{-1} & , & a_{p,1}^q=0\\
                                     (\alpha_p\,\alpha_q)^2 & , & a^q_{p,1}=1
                                    \end{array}\right.
        \end{equation*}
\end{theorem}

\begin{proof}
 The proof uses \cite[Theorem 3.2]{umafundamental}. We can relate the above notations with those in the proof of \cite[Theorem 3.2]{umafundamental} as follows : $$y_{j,\varepsilon}=t_\varepsilon\cdot\alpha_j\cdot t_\varepsilon$$ for $1\leq j\leq k$ and $\varepsilon\in\mathbb Z_2^n$. Thus $y_{j,\varepsilon}=\alpha_j^{\pm1}$ by (\ref{yepsilon}). Then it can be seen that 
 \begin{equation*}
 \begin{split}
  y_{p,\varepsilon}\cdot y_{q,B^p}\cdot y_{p,C^{p,q}}\cdot y_{q,B^q} & =t_\varepsilon \cdot\alpha_p\cdot t_{\mathbf a_p}\alpha_q t_{\mathbf a_p}\cdot t_{\mathbf a_q+\mathbf a_p}\alpha_p t_{\mathbf a_q+\mathbf a_p}\cdot t_{\mathbf a_q}\alpha_qt_{\mathbf a_q}\cdot t_{\varepsilon}\\
  ~ & = : X^\varepsilon_{p,q}
  \end{split}
 \end{equation*}
 whenever $1\leq p,q\leq k$. By exchanging the rolls of $p$ and $q$ we get that $X_{p,q}^\varepsilon$ is conjugate to $X_{q,p}^\varepsilon$ or $(X_{q,p}^\varepsilon)^{-1}$ by an element from the free group generated by $S$. Thus we may assume that $p<q$. We now look at the following cases :
 \begin{enumerate}
  \item If $p,q\leq l$ then $X_{p,q}^\varepsilon=t_\varepsilon\cdot\alpha_p\,\alpha_q\,\alpha_p\,\alpha_q\cdot t_\varepsilon =(\alpha_p\,\alpha_q)^2$ .
  \item If $p\leq l$ and $q>l$ then
  \begin{equation*}
  X_{p,q}^\varepsilon=\left\{\begin{array}{lcl}
                                     t_\varepsilon\cdot\alpha_p\,\alpha_q\,\alpha_p\,\alpha_q^{-1}\cdot t_\varepsilon & , & a_{p,1}^q=0\\
                                     t_\varepsilon\cdot(\alpha_p\,\alpha_q^{-1})^2\cdot t_\varepsilon & , & a_{p,1}^q=1
                                    \end{array}\right.
  \end{equation*}
  \item If $p,q> l$ then
  \begin{equation*}
   X_{p,q}^\varepsilon=\left\{\begin{array}{ccl}
                             t_\varepsilon\cdot\alpha_p\,\alpha_q\,\alpha_p^{-1}\,\alpha_q^{-1}\cdot t_\varepsilon & , & a_{p,1}^q=0\\
                             t_\varepsilon\cdot\alpha_p\,\alpha_q^{-1}\,\alpha_p^{-1}\,\alpha_q^{-1}\cdot t_\varepsilon & , & a_{p,1}^q=1
                            \end{array}\right.
  \end{equation*}
  \end{enumerate}
  From $(\ref{yepsilon})$ it is easy to see that $X_{p,q}^\varepsilon$ in (2) is conjugate to either $x_{p,q}^\prime$ or $(x_{p,q}^\prime)^{-1}$ and $X_{p,q}^\varepsilon$ in (3) is conjugate to either $x_{p,q}$ or $x_{p,q}^{-1}$ by an element from the free group generated by $S$. Also by definition of $\alpha_j$ it is clear that $\alpha_j^2=1$ for $1\leq j\leq l$. Thus it follows that $R$ gives a complete set of relations of $\pi_1(B_k)$. Hence the theorem.
\end{proof}

\begin{remark}
 Alternately, the presentation for $\pi_1(B_k)$ can be derived from (\ref{ses}).
\end{remark}

\begin{corollary}\label{abelian}
 The fundamental group $\pi_1(B_k)$ is abelian if and only if $B_k\simeq M^l\times(S^1)^{k-l}$ for some $0\leq l\leq k$.
\end{corollary}

\begin{proof}
 From Theorem \ref{fungp} it is easy to see that $\pi_1(B_k)$ is abelian if and only if for all $1\leq p\leq k$ , $l+1\leq q\leq k$ and $p<q$ , $a_{p,1}^q=0$, where $l$ is as in Remark \ref{l}. Hence the corollary follows.
\end{proof}

\begin{lemma}\label{aspherical}
 $B_k$ is an aspherical manifold if and only if it is a real Bott manifold.
\end{lemma}
\begin{proof}
 We know from \cite[Theorem 6.1]{umafundamental} that a real toric variety is aspherical if and only if the corresponding fan is flag-like. A fan is flag-like if and only if for every collection $S$ of edge vectors of the fan if $\langle u,u^\prime\rangle$ spans a cone for all $u,u^\prime\in S$ then $S$ spans a cone. Now if $n_i\geq2$ for some $i$ then the set of edge vectors $\{ u^i_{l_i}\}_{l_i=0}^{n_i}$ of $\Sigma_k$ is such that $\langle u^i_{l_i},  u^i_{l_i^\prime}\rangle$ spans a cone in $\Sigma_k$ but $\langle u^i_0,\cdots,u^i_{n_i}\rangle$ does not. Conversely, if $n_i=1$ for all $i$ then it can be easily seen that $\Sigma_k$ is flag-like. The theorem follows. 
\end{proof}

\begin{corollary}
 The group $\pi_1(B_k)$ is torsion free if and only if $B_k$ is a real Bott manifold.
\end{corollary}
\begin{proof}
 If $B_k$ is a real Bott manifold then by Lemma \ref{aspherical} it is aspherical and hence $\pi_1(B_k)$ is torsion free. Conversely, if $n_i\neq1$ for some $i$, then from Theorem \ref{fungp}, $\alpha_i\neq1$ but $\alpha_i^2=1$, so that $\alpha_i$ is a torsion element.
\end{proof}

\subsection{Further group theoretic properties of $\pi_1(B_k)$}\label{FGTP}

\begin{proposition}\label{sol}
 The commutator subgroup $[\pi_1(B_k),\pi_1(B_k)]$ is abelian. In particular $\pi_1(B_k)$ is a solvable group.
\end{proposition}
\begin{proof}
By \cite[Lemma 4.1]{umafundamental} we know that $[W,W]$ is abelian if and only if for each edge $u$ in the fan $\Sigma_k$, there is at most one edge $u^\prime$ such that $\langle u,u^\prime\rangle$ does not span a cone in $\Sigma_k$. We have $\langle u^i_{l_i},u^j_{l_j^\prime}\rangle$ spans a cone in $\Sigma_k$ whenever $i\neq j$ or $i=j$ , $n_i\geq2$. When $i=j$ and $n_i=1$ there are only two edge vectors $u^i_0$ and $u^i_1$ and they do not span a cone in $\Sigma_k$. Hence $[W,W]$ is abelian. Since $\pi_1(B_k)$ is a subgroup of $W$ it follows that $[\pi_1(B_k),\pi_1(B_k)]$ is abelian. Finally, $1\leq[\pi_1,\pi_1]\leq\pi_1(B_k)$ gives an abelian tower for $\pi_1(B_k)$ so that $\pi_1(B_k)$ is solvable.
\end{proof}

Let $\bar{\alpha_j}$ denote the image of $\alpha_j$ under the canonical abelianization homomorphism 
\begin{equation}\label{H1} \pi_{1}(B_k)\rightarrow H_1(B_k;\mathbb Z)\simeq\pi_{1}(B_k)/[\pi_1(B_k),\pi_1(B_k)]. \end{equation}

We then have the following description of $H_{1}(B_k;\mathbb Z)$.

\begin{proposition}\label{btfirsthomology} The group $H_1(B_k;\mathbb Z)$ has a presentation with generators 
 \begin{equation}\label{gens}
 \langle \bar{\alpha_j} : 1\leq j\leq k \rangle
 \end{equation} 
 and relations as follows :
\begin{equation}\label{rel1}
\bar{\alpha_p}\cdot\bar{\alpha_q}\cdot\bar{\alpha_p}^{-1}\cdot\bar{\alpha_q}^{-1}
\end{equation}
for $1\leq p,q\leq k$ and 
\begin{equation}\label{rel2} 
\bar{\alpha_q}^2
\end{equation}
for $1\leq q\leq l$ and for those $l<q\leq k$ for which there exists a $p<q$ such that $a_{p,1}^q=1$. Thus additively we have an isomorphism $H_{1}(B_k;\mathbb Z)\simeq \mathbb Z^{k-r-l}\bigoplus \mathbb Z_2^{l+r}$ where $r$ is the number of $l<q\leq k$ as above. 
\end{proposition}

\begin{proof}
 The proof follows readily from (\ref{H1}) and (\ref{rels}).
\end{proof}
%
 
\begin{proposition}\label{nil}
 The group $\pi_1(B_k)$ is nilpotent if and only if it is abelian. 
 \end{proposition}
 \begin{proof}
  Observe by (\ref{W}) that \begin{equation}\label{aq2}\alpha_j^2=(s_{j,0}\cdot s_{1,1}^{a_{j,1}^1}\cdot s_{1,2}^{a_{j,2}^1}\cdots s_{1,n_1}^{a_{j,n_1}^1}\cdots s_{k,1}^{a_{j,1}^k}\cdots s_{k,n_k}^{a_{j,n_k}^k})^2=(s_{j,0}\cdot s_{j,1})^2\neq1\end{equation} for $j>l$ and $\alpha_j^2=1$ for $j\leq l$. If $\pi_1(B_k)$ is not abelian then by Theorem \ref{fungp}, there exists a $q>l$ and a $p<q$ such that $a_{p,1}^q=1$. Then it follows from Proposition \ref{btfirsthomology} that $\alpha_q^2\in[\pi_1(B_k),\pi_1(B_k)]$. Further, by (\ref{rels}), $\alpha_p\,\alpha_q^{-2}\,\alpha_p^{-1}\,\alpha_q^{2}=(\alpha_p\,\alpha_q^{-1}\,\alpha_p^{-1})^2\,\alpha_q^{2}=\alpha_q^{2}\,\alpha_q^{2}=\alpha_q^{4}\in[\pi_{1}(B_k),[\pi_1(B_k),\pi_1(B_k)]]$. Proceeding similarly by induction we get that \[\alpha_p\,\alpha^{-2^m}_q\,\alpha_p^{-1}\,\alpha^{2^m}_q=\alpha_q^{2^{m+1}}\in\pi_{1}(B_k)^{(m)}.\] Here $\pi_1(B_k)^{(1)}:=[\pi_1(B_k),\pi_1(B_k)]$ and $\pi_{1 }(B_k)^{(m)}:=[\pi_1(B_k), \pi_1(B_k)^{(m-1)}]$. 
  Finally, by (\ref{aq2}) and the fact that $s_{q,0}\cdot s_{q,1}$ is of infinite order in $W$, we have that $\alpha_q^{2^m}=(s_{q,0}\cdot s_{q,1})^{2^m}\neq1$ for any $m$. The proposition follows.  
 \end{proof}
 
\begin{remark}
 Here we mention that the fundamental group of a real toric variety is not in general solvable. For example, the non-orientable surfaces of genus $g$ are real toric varieties (see for example \cite[Section 4.5, Remark 4.5.2]{jurkiewiczthesis} and \cite[Remark 3.3]{umafundamental}), whose fundamental groups contain free subgroups of rank $g-1$ (see for example \cite[p. 62, Section 4]{bmms}). Thus whenever $g\geq 3$, these groups are not solvable. 
\end{remark}

\subsection{Remark on higher homotopy groups} 
\begin{theorem}
 The higher homotopy groups $\pi_j(B_k)$ for $j\geq2$ are isomorphic to $\pi_j(\mathbb P^{n_1})\times\cdots\times\pi_j(\mathbb P^{n_k})$.
\end{theorem}

\begin{proof}
 Note that the fibration $\mathbb P^{n_k}\xhookrightarrow{i} B_k\xrightarrow{p_k} B_{k-1}$ induces a long exact sequence of homotopy groups
\begin{equation}\label{les}
 \cdots \rightarrow\pi_{j}(\mathbb P^{n_k})\xrightarrow{i_*}\pi_{j}(B_k)\xrightarrow{(p_k)_*}\pi_j(B_{k-1})\xrightarrow{\partial_*}\pi_{j-1}(\mathbb P^{n_k})\rightarrow\cdots
\end{equation}

Since $B_k=\mathbb P(\mathbf 1\oplus L_1\oplus\cdots\oplus L_{n_k})$ the zero sections of each $L_i\rightarrow B_{k-1}$ for $1\leq i\leq n_k$ define a section of $p_k:B_k\rightarrow B_{k-1}$, namely the map $s_k:B_{k-1}\rightarrow B_k$ that sends any point to $[(1,0,\cdots,0)]$ in the fiber above that point. The induced map, $(s_k)_*:\pi_j(B_{k-1})\rightarrow\pi_j(B_k)$ is a right inverse of $(p_k)_*$. This in particular, makes all the maps $(p_k)_*$ surjective so that the maps $\partial_*$ are identically $0$ and hence the maps $i_*$ are injective. Thus we have the split exact sequence
\begin{equation}\label{ses}
 0\rightarrow\pi_j(\mathbb P^{n_k})\xrightarrow{i_*}\pi_j(B_k)\xrightarrow{(p_k)_*}\pi_j(B_{k-1})\rightarrow0.
\end{equation}

Since $\pi_j(B_k)$ is abelian for $j\geq2$, by induction on $k$ we get that $\pi_j(B_k)\simeq\pi_j(\mathbb P^{n_1})\times\cdots\times\pi_j(\mathbb P^{n_k})$.
\end{proof}

\section{Orientability and Spin structure}\label{4}

The results in this section extend those in \cite[Section 3]{umaraisaspin}. We begin by recalling the presentation of the cohomology ring with $\mathbb Z_2$ coefficients of $B_k$. We then give a recursive formula for the total Stiefel-Whitney class and closed formulae for the first and second Stiefel-Whitney classes of $B_k$. We hence obtain necessary and sufficient conditions for orientability and spin structure on $B_k$ in terms of certain identities on the entries $a_{i,l_j}^j$ of the matrix $A$.

\begin{proposition}\label{cohoring}
 Let $\mathcal R$ be the ring $\mathbb Z_2[x_{10},\cdots,x_{1n_1},\cdots,x_{k0},\cdots,x_{kn_k}]$ and let $\mathcal I$ denote the ideal in $\mathcal R$ generated by the following set of elements 
 \begin{equation}\label{ideal}
  \{x_{i0}\,x_{i1}\cdots x_{in_i}\,,\,x_{il_i}+\sum_{j=1}^ia_{jl_i}^i\,x_{j0}\,\forall\, 1\leq i\leq k\,,\,1\leq l_i\leq n_i\}.
 \end{equation}
 As a graded $\mathbb Z_2$-algebra $H^*(B_k;\mathbb Z_2)$ is isomorphic to $\mathcal{R/I}$.
\end{proposition}

\begin{proof}
 Since $B_k$ is the toric variety associated to the smooth projective fan $\Sigma_k$ the result follows from \cite[Theorem 4.3.1]{jurkiewiczthesis}. Alternately, since $B_k$ is the small cover over the polytope $\prod_{i=1}^k\Delta^{n_i}$ with characteristic function $\lambda$, the result follows from \cite[Theorem 4.14]{davisjanus}. 
\end{proof}

Let $w_p(B_k)$ denote the $p$th Stiefel-Whitney class of $B_k$ for $0\leq p\leq n$ with the understanding that $w_0(B_k)=1$. Then $w(B_k)=1+w_1(B_k)+\cdots+w_n(B_k)$ is the total Stiefel-Whitney class of $B_k$.

\begin{theorem}\label{totswc}
 The following holds in the $\mathbb Z_2$-algebra $\mathcal {R/I}$ :
 \begin{equation}\label{rft}
  w(B_k)=w(B_{k-1})\cdot\left(\prod_{l_k=1}^{n_k}(1+\sum_{j=1}^ka_{j,l_k}^k\,x_{j0})\right)\cdot(1+x_{k0}).
 \end{equation}
\end{theorem}

\begin{proof}
 From \cite[Corollary 6.8]{davisjanus} for $B_k$ we have 
 \begin{equation}\label{sw}
  w(B_k)=\prod_{i=1}^k\prod_{l_i=0}^{n_i}(1+x_{il_i}).
 \end{equation}
 Note that the defining matrix for $B_{k-1}$ is the $(k-1)\times (n-n_k)$ submatrix of $A$ obtained by deleting the $k$th row and the last $n_k$ columns of $A$. Moreover, via pullback along $p_k^*$, where $p_k$ is the map in (\ref{bt}), $H^*(B_{k-1};\mathbb Z_2)$ can be identified with the subring $\mathcal{R'/I'}$ of $\mathcal{R/I}$ where $\mathcal R'=\mathbb Z_2[x_{10},\cdots,x_{1n_1},\cdots, x_{k-1\ 0},\cdots, x_{k-1\ n_{k-1}} ] $ and $\mathcal I' $ is the ideal generated by the relations 
 \begin{equation}\label{st1}
   \{x_{i0}\,x_{i1}\cdots x_{in_i}~,~ x_{il_i}+\sum_{j=1}^{i}a_{j,l_i}^i\cdot x_{j0} ~~\mbox{for}~~ 1\leq i\leq k-1\,,\,1\leq l_i\leq n_i\}.
 \end{equation} 
 Since $B_k$ is a $\mathbb P^{n_k}$-bundle over $B_{k-1}$, we further have the following presentation of $H^*(B_k;\mathbb Z_2)$ as an algebra over $H^*(B_{k-1};\mathbb Z_2)$: 
 \begin{equation}\label{presind}
  H^*(B_k;\mathbb Z_2)\simeq H^*(B_{k-1};\mathbb Z_2)[x_{k0},\cdots,x_{k n_k}]/ J
 \end{equation}
 where $J$ is the ideal generated by the relations 
 \begin{equation}\label{relnind}
  x_{k0}\,x_{k1}\cdots x_{kn_k}, ~~x_{kl_k}+\sum_{j=1}^{k}a_{j,l_k}^kx_{j0} ~~\mbox{for}~~ 1\leq l_k\leq n_k.
 \end{equation}
 Furthermore, via $p_k^*$ we can identify $w(B_{k-1})$ with the expression :
 \begin{equation}\label{sw1}
  w(B_{k-1})=\prod_{i=1}^{k-1}\prod_{l_i=0}^{n_i}(1+x_{il_i}).
 \end{equation}
 in $\mathcal {R'}$ where $x_{il_i}$ for $1\leq i\leq k-1$ and $1\leq i\leq n_i$ satisfy the relations (\ref{st1}). Now by (\ref{sw}) and (\ref{sw1}), it follows that,
 \begin{equation}
  w(B_k)=w(B_{k-1})\cdot\prod_{l_k=0}^{n_k}(1+x_{kl_k}).
 \end{equation}
 Using the relations in $\mathcal I$ we get (\ref{rft}).
 Hence the theorem.
 \end{proof}

\begin{corollary}\label{w1result}
 We have the following expression for $w_1(B_k)$ in $\mathcal {R/I}$ :
 \begin{equation}\label{w1}
  w_1(B_k)=\sum_{j=1}^k(1+\sum_{i=1}^k\sum_{l_i=1}^{n_i}a_{j,l_i}^i)\cdot x_{j0}.
 \end{equation}
\end{corollary}

\begin{proof}
 Note that under the isomorphism of graded algebras $H^*(B_k;\mathbb Z_2)$ and $\mathcal{R/I}$, $w_1(B_k)\in H^1(B_k;\mathbb Z_2)$ corresponds to a polynomial of degree $1$ in $x_{il_i} ,1\leq i\leq k , 1\leq l_i\leq n_i$ modulo $\mathcal I$.  Thus by comparing the degree $1$-terms on either side of (\ref{rft}) and using induction on $k$ the lemma follows.

\end{proof}

\begin{corollary}\label{w2result}
 We have the following expression for $w_2(B_k)$ in $\mathcal {R/I}$ :
 \begin{equation}\label{w2}
   w_2(B_k)=\sum_{s=1}^l T_s^\prime\cdot x_{s0}^2 + \sum_{\substack{r,s=1\\ r<s}}^l T_{rs}^\prime\cdot x_{r0}\cdot x_{s0} + \sum_{\substack{1 \leq r\leq k\\ l<s\leq k\\ r<s}}(T_{rs}^\prime+a_{r,1}^s\cdot T_s^\prime)\cdot x_{r0}\cdot x_{s0}
 \end{equation}
 Where \begin{equation}
        T_s^\prime=\sum_{\substack{1\leq i\leq k\\ 1\leq l_i\leq n_i}}a_{s,l_i}^i+\sum_{\substack{1\leq i\leq k\\ 1\leq l_i<l_i'\leq n_i}}a_{s,l_i}^i\,a_{s,l_i'}^i+\sum_{\substack{1\leq i<j\leq k\\ 1\leq l_i\leq n_i\\ 1\leq l_j\leq n_j}}a_{s,l_i}^i\,a_{s,l_j}^j
       \end{equation}
 and \begin{equation}
      T_{rs}'= 1+ \sum_{\substack{1\leq i\leq k\\ 1\leq l_i\leq n_i}}(a_{r,l_i}^i+a_{s,l_i}^i)+\sum_{\substack{1\leq i\leq k\\ 1\leq l_i,l_i'\leq n_i\\ l_i\neq l_i'}}a_{r,l_i}^i\,a_{s,l_i'}^i+\sum_{\substack{1\leq i<j\leq k\\ 1\leq l_i\leq n_i\\ 1\leq l_j\leq n_j}}a_{r,l_i}^i\,a_{s,l_j}^j 
     \end{equation}
\end{corollary}

\begin{proof}
 Using the relations (\ref{ideal}) in (\ref{sw}), $w(B_k)$ can be identified with the class of the following term in $\mathcal{R/I}$ :
 \begin{equation}\label{swex}
  \prod_{i=1}^k(1+x_{i0})\cdot\prod_{i=1}^k\prod_{l_i=1}^{n_i}\left(1+\sum_{j=1}^k a_{j,l_i}^i\cdot x_{j0}\right)
 \end{equation}
 
 (Here we note that $a_{j,l_i}^i=0$ for $j>i$.) Since Proposition \ref{cohoring} gives an isomorphism of graded $\mathbb Z_2$-algebras, the degree 2 term of $w(B_k)$, namely $w_2(B_k)$ can be identified with the degree $2$ term of the expression (\ref{swex}) which is the class of the following term in $\mathcal{R/I}$ :
 \begin{equation}\label{w2pre}
   \sum_{s=1}^k {T_s'}\cdot\, x_{s0}^2\ + \sum_{\substack{r,s=1\\ r<s}}^k {T_{rs}'}\cdot\, x_{r0}\cdot x_{s0}
 \end{equation} 
 For $l\leq s\leq k$, we can use the relations in (\ref{ideal}) to obtain,
 \begin{equation}\label{sqfree}
  x_{s0}^2=\sum_{j=1}^{s-1} a_{j,1}^s\cdot x_{j0}\cdot x_{s0}
 \end{equation}
 Using (\ref{sqfree}) in (\ref{w2pre}) we get (\ref{w2}). Hence the corollary.
\end{proof}

\begin{theorem}\label{orientable}
 The space $B_k$ is orientable if and only if
 \begin{equation}\label{orieqn}
  \sum_{i=1}^k\sum_{l_i=1}^{n_i}a_{j,l_i}^i\equiv 1\bmod2 \text{ for all }1\leq j\leq k.
 \end{equation}
\end{theorem}

\begin{proof}
 Note that via the isomorphism in Theorem \ref{cohoring} of graded algebras $\mathcal{R/I}$ and $H^*(B_k;\mathbb Z_2)$, $\{x_{10},\cdots,x_{k0}\}$ corresponds to a basis over $\mathbb Z_2$ of $H^1(B_k;\mathbb Z_2)$. Thus by (\ref{w1}) it follows that $w_1(B_k)=0$ if and only if 
 \begin{equation}
    \sum_{i=1}^k\sum_{l_i=1}^{n_i}a_{j,l_i}^i\equiv 1\bmod2 \text{ for all }1\leq j\leq k.
 \end{equation}
 Furthermore, since a necessary and sufficient condition for a compact connected differentiable manifold $M$ to be orientable is $w_1(M)=0$, the theorem follows.
\end{proof}

\begin{remark}
 Note that Theorem \ref{orientable} also follows as a direct consequence of \cite[Theorem 1.7]{naknishiorientability}.
\end{remark}

\begin{theorem}\label{spin}
 Let \begin{equation}\label{term1} T_s:=\sum_{\substack{1\leq i\leq k\\ 1\leq l_i<l_i'\leq n_i}}a_{s,l_i}^i\,a_{s,l_i'}^i+\sum_{\substack{1\leq i<j\leq k\\ 1\leq l_i\leq n_i\\ 1\leq l_j\leq n_j}}a_{s,l_i}^i\,a_{s,l_j}^j \end{equation}
 and \begin{equation}\label{term2} T_{rs}:=\sum_{\substack{1\leq i\leq k\\ 1\leq l_i\leq n_i}}a_{r,l_i}^i\,a_{s,l_i}^i. \end{equation}
 The orientable generalized real Bott manifold $B_k$ admits a spin structure if and only if the following identities hold :
 \begin{equation} T_s\equiv1\bmod2 \qquad \text{for}\quad 1\leq s\leq l\end{equation} 
 \begin{equation} T_{rs}\equiv0\bmod2 \qquad \text{for}\quad 1\leq r<s\leq l\end{equation}
 \begin{equation} \label{term3} T_{rs}+a_{r,1}^s\cdot(1+ T_s)\equiv0\bmod2\qquad \text{for}\quad 1\leq r\leq k\ ,\ l<s\leq k. \end{equation}
\end{theorem}

\begin{proof}
  Using (\ref{orieqn}) we can simplify $T_s'$ and $T_{rs}'$ from Corollary \ref{w2result} as follows,
 \begin{equation}\label{ts}
  T_s'=1+\sum_{\substack{1\leq i\leq k\\ 1\leq l_i<l_i'\leq n_i}}a_{s,l_i}^i\,a_{s,l_i'}^i+\sum_{\substack{1\leq i<j\leq k\\ 1\leq l_i\leq n_i\\ 1\leq l_j\leq n_j}}a_{s,l_i}^i\,a_{s,l_j}^j=1+T_s
 \end{equation}
 \begin{equation}\label{trs}
 \begin{split}
  T_{rs}' & = 1+ \sum_{\substack{1\leq i\leq k\\ 1\leq l_i,l_i'\leq n_i\\ l_i\neq l_i'}}a_{r,l_i}^i\,a_{s,l_i'}^i+\sum_{\substack{1\leq i<j\leq k\\ 1\leq l_i\leq n_i\\ 1\leq l_j\leq n_j}}a_{r,l_i}^i\,a_{s,l_j}^j\\
   ~& =1+ \sum_{\substack{1\leq i\leq k\\ 1\leq l_i\leq n_i}}a_{r,l_i}^i\cdot(1+a_{s,l_i}^i)=1+1+T_{rs}=T_{rs}
  \end{split}
 \end{equation}
 Using (\ref{ts}), (\ref{trs}) in (\ref{w2}), $w_2(B_k)$ can be identified with the class of the following term in $\mathcal{R/I}$ :
 \begin{equation}\label{spineqn}
 \sum_{s=1}^l (1+T_s)\cdot x_{s0}^2 + \sum_{\substack{r,s=1\\ r<s}}^l T_{rs}\cdot x_{r0}\cdot x_{s0} + \sum_{\substack{1 \leq r\leq k\\ l<s\leq k\\ r<s}}(T_{rs}+a_{r,1}^s\cdot (1+T_s))\cdot x_{r0}\cdot x_{s0}
 \end{equation}
 By Proposition \ref{cohoring}, as a $\mathbb Z_2$-vector space, $H^2(B_k;\mathbb Z_2)$ is isomorphic to the subspace of $\mathcal{R/I}$  generated by the classes of $x_{s0}^2$ for $1\leq s\leq k$ and $x_{r0}\cdot x_{s0}$ for $1\leq r<s\leq k$. Further, since $B_k$ is constructed as an iterated sequence of projective bundles, the mod 2 Poincar\'{e} polynomial of $B_k$ has the following expression 
 \begin{equation}\label{poincare} P_t(B_k;\mathbb Z_2)=\prod_{i=1}^kP_t(\mathbb P^{n_i};\mathbb Z_2)=\prod_{i=1}^k(1+t+\cdots+t^{n_i}).\end{equation} 
 It follows from (\ref{poincare}) and (\ref{sqfree}) that $H^2(B_k;\mathbb Z_2)$ is freely generated by the classes of $x_{s0}^2$ for $1\leq s\leq l$ and $x_{r0}\cdot x_{s0}$ for $1\leq r<s\leq k$.
 Moreover, the necessary and sufficient condition for an orientable manifold $M$ to admit a spin structure is $w_2(M)=0$. Thus the theorem follows from (\ref{spineqn}).
\end{proof}

\subsection{{The case of a real Bott manifold}}

\begin{corollary}
  When $l=0$, the orientable real Bott manifold $B_k$ is Spin if and only if 
 \begin{equation}\label{btspin} T_{rs}+a_{r,1}^s\cdot(1+T_s)\equiv0\bmod2 \quad \text{ for all }\quad 1\leq r<s\leq k\end{equation}
 where $T_{rs}$ and $T_s$ are as in (\ref{term2}) and (\ref{term1}) respectively. Further, if we let $C:=A-I$ (where $A$ be the $k\times k$ upper triangular matrix associated to $B_k$ and $I$ is the $k\times k$ identity matrix) then (\ref{btspin}) is equivalent to 
 \begin{equation}\label{ourbtspin}
  \sum_{p=1}^k c_{r,p}\,c_{s,p} + c_{r,s}\cdot\sum_{\substack{p,q=1\\ p<q}}^k c_{r,p}\,c_{s,q}\equiv0\bmod2 \quad\text{ for all }\quad 1\leq r<s\leq k
 \end{equation}
 where $c_{i,j}$ are the entries of $C$.
 In particular, \cite[Theorem 3.2]{umaraisaspin} and, \cite[Theorem 1.2]{gasiorspin} follow.
 \end{corollary}

\begin{proof}
 The first statement follows immediately from Theorem \ref{spin} by putting $l=0$. The left hand side of (\ref{btspin}) is 
 \begin{equation}\label{LHS}
  \sum_{p=1}^k a_{r,1}^p\,a_{s,1}^p + a_{r,1}^s\cdot\big(1+\sum_{\substack{p,q=1\\ p<q}}a_{s,1}^p\,a_{s,1}^q\big)
 \end{equation}
 
 Since $c_{i,j}=a_{i,1}^j$ when $i\neq j$ and $c_{i,i}=a_{i,1}^i-1=0$, (\ref{LHS}) reduces to,
 \begin{equation*}
  \sum_{\substack{p=1\\ p\neq r,s}}^k c_{r,p}\,c_{s,p} + a_{r,1}^r\,a_{s,1}^r + a_{r,1}^s\,a_{s,1}^s + c_{r,s}\cdot\Big(1+\sum_{\substack{p,q=1\\ p<q , p,q\neq s}}^kc_{s,p}\,c_{s,q}+a_{s,1}^s\cdot\sum_{\substack{i=1\\ i\neq s}}^ka_{s,1}^i\Big).
  \end{equation*}
  
  Further, using (\ref{orieqn}) the above expression can be rewritten as 
  \begin{equation*}
   \sum_{\substack{p=1\\ p\neq r,s}}^k c_{r,p}\,c_{s,p} + 1\cdot 0 + c_{r,s}\cdot1 + c_{r,s}\cdot\Big(1+\sum_{\substack{p,q=1\\ p<q , p,q\neq s}}^kc_{s,p}\,c_{s,q}+1\cdot(1+a_{s,1}^s)\Big).
  \end{equation*}
  
  This further can be simplified to $$\sum_{p=1}^k c_{r,p}\,c_{s,p} + c_{r,s}\cdot\sum_{\substack{p,q=1\\ p<q}}^k c_{r,p}\,c_{s,q}.$$
 
  Thus (\ref{btspin}) is equivalent to (\ref{ourbtspin}). Hence \cite[Theorem 3.2]{umaraisaspin} follows. This also implies \cite[Theorem 1.2]{gasiorspin} by \cite[Theorem 3.10, Corollary 3.11]{umaraisaspin}.
\end{proof}

\subsection{Digraph interpretation}
 Choi, Masuda and Suh in \cite{masudachoidigraph} associate an acyclic directed graph to a real Bott manifold $M(A)$ by viewing the associated matrix $A$ as the adjacency matrix of the digraph. In \cite[Lemma 4.1]{masudachoidigraph}, they give necessary and sufficient conditions on the digraph for $M(A)$ to be orientable and to admit a symplectic form. In \cite[Theorem 4.5]{umaraisaspin} the authors give a necessary and sufficient condition on the digraph for $M(A)$ to admit a spin structure. 
 
 One can similarly associate a \emph{labeled multidigraph} to a generalized real Bott manifold. Let $B_k$ be the generalized real Bott manifold associated to the vector matrix $A$ from (\ref{UTmatrix}). Define $D_{B_k}$, as the graph with $k$ vertices $\{w_1,\cdots,w_k\}$ and edges as follows : there is an edge from $w_i$ to $w_j$ with label $a_{i,l_j}^j\in\{0,1\}$. Also for every $i$ there are $n_i$ loops (edges beginning and ending at the same vertex), each with label $1$, based at $w_i$. The \emph{out-degree} of a vertex $w_i$ is the number of edges with label $1$ that begin at $w_i$. It can be seen that $B_k$ is orientable if and only if every vertex in $D_{B_k}$ has odd out-degree. One can similarly make interpretations on $D_{B_k}$ for $B_k$ to admit a spin  structure. The authors are trying to find such an interpretation that generalizes the one for real Bott manifolds.

\begin{eg}\hfill
 \begin{enumerate}
  \item Let $P=\Delta^2\times\Delta^1$. The possible $2$-step generalized Bott manifolds over $P$ correspond to the  characteristic matrices given by $$A=\left(\begin{array}{cc|c}
                                                a_{1,1}^1 & a_{1,2}^1 & a_{1,1}^2\\
                                                a_{2,1}^1 & a_{2,2}^1 & a_{2,1}^2
                                             \end{array}\right)=\left(\begin{array}{cc|c}
                                                                        1 & 1 & a_{1,1}^2\\
                                                                        0 & 0 & 1
                                                                       \end{array}\right)$$
   where $a_{1,1}^2=0$ or $1$. From Theorem \ref{orientable} $B_2=M(A)$ is orientable if and only if $a_{1,1}^1 + a_{1,2}^1 + a_{1,1}^2=1+1+a_{1,1}^2\equiv1\bmod2$ and $a_{2,1}^1 + a_{2,2}^1 + a_{2,1}^2=0+0+1\equiv1\bmod2$. Thus $B_2$ is orientable if and only if $a_{1,1}^2=1$.\\
   Further, $T_1=a_{1,1}^1\,a_{1,2}^1+a_{1,1}^1\,a_{1,1}^2+a_{1,2}^1\,a_{1,1}^2=1+1+1\equiv1\bmod2$ and\\ $T_{12}+a_{1,1}^2\cdot(1+T_2)=1+1(1+0)\equiv0\bmod2$. So by Theorem \ref{spin} $B_2$ is also spin if and only if $a_{1,1}^2=1$.
   
   \item Let $P=\Delta^2\times\Delta^2$. The possible $2$-step generalized Bott manifolds over $P$ correspond to the  characteristic matrices given by $$A=\left(\begin{array}{cc|cc}
                                                a_{1,1}^1 & a_{1,2}^1 & a_{1,1}^2 & a_{1,2}^2\\
                                                a_{2,1}^1 & a_{2,2}^1 & a_{2,1}^2 & a_{2,2}^2
                                             \end{array}\right)=\left(\begin{array}{cc|cc}
                                                                        1 & 1 & a_{1,1}^2 & a_{1,2}^2\\
                                                                        0 & 0 & 1 & 1
                                                                       \end{array}\right)$$
    where $a_{1,1}^2\,,\,a_{1,2}^2\in\{0,1\}$. $B_2=M(A)$ is orientable if and only if $a_{1,1}^2=1$ and $a_{1,2}^2=0$ or $a_{1,1}^2=0$ and $a_{1,2}^2=1$. In both these cases $T_{12}=a_{1,1}^2+a_{1,2}^2\equiv1\bmod2$ so that $B_2$ is not spin.
    
    \item Let $P=\Delta^2\times\Delta^1\times\Delta^1$. The possible $3$-step generalized Bott manifolds over $P$ correspond to the  characteristic matrices given by $$A=\left(\begin{array}{cc|c|c}
                                                a_{1,1}^1 & a_{1,2}^1 & a_{1,1}^2 & a_{1,1}^3\\
                                                a_{2,1}^1 & a_{2,2}^1 & a_{2,1}^2 & a_{2,1}^3\\
                                                a_{3,1}^1 & a_{3,2}^1 & a_{3,1}^2 & a_{3,1}^3
                                             \end{array}\right)=\left(\begin{array}{cc|c|c}
                                                                        1 & 1 & a_{1,1}^2 & a_{1,1}^3\\
                                                                        0 & 0 & 1 & a_{2,1}^3 \\
                                                                        0 & 0 & 0 & 1
                                                                        \end{array}\right)$$
    where $a_{1,1}^2\,,\,a_{1,1}^3\,,\,a_{2,1}^3\in\{0,1\}$. Then $B_3=M(A)$ is orientable if and only if $(a_{1,1}^2\,,\,a_{1,1}^3\,,\,a_{2,1}^3)=(1,0,0)$ or $(a_{1,1}^2\,,\,a_{1,1}^3\,,\,a_{2,1}^3)=(0,1,0)$. In both cases we see that $T_1=1+a_{1,1}^2+a_{1,1}^3+a_{1,1}^2\,a_{1,1}^3\equiv0\bmod2$ so that $B_3$ is not spin.
 \end{enumerate}
 
\end{eg}

\begin{acknow}
 The first author wishes to thank University Grants Commission (UGC), India for financial assistance.
\end{acknow}

\bibliographystyle{siam}
\bibliography{refs_raisa}

\end{document}